\documentclass[reqno,12pt]{amsart}

\usepackage[T1]{fontenc}
\usepackage{amssymb,amsfonts,amsbsy,graphicx,setspace,color}
\usepackage{mathtools,hyperref}
\usepackage{enumerate}
\usepackage{tikz}
\usetikzlibrary{hobby}
\usetikzlibrary{arrows,calc, decorations.markings, intersections,through,backgrounds}
\usepackage{tkz-euclide}

\makeatletter

\everymath{\displaystyle} 

\theoremstyle{plain}
\newtheorem{thm}{Theorem}[section]  
\newtheorem{lmm}[thm]{Lemma}  
\newtheorem{crl}[thm]{Corollary}
\numberwithin{equation}{section} 
\numberwithin{figure}{section} 
\newtheorem{prp}[thm]{Proposition}
\theoremstyle{remark}
\newtheorem{rmk}[thm]{Remark}
\newtheorem{xpl}[thm]{Example}
\theoremstyle{definition}
\newtheorem{dfn}[thm]{Definition}

\newcommand{\leb}{\mathrm{Leb} \,}
\newcommand{\interior}{\mathrm{int} \,}

\newcommand{\supp}{\mathrm{supp} \,}
\newcommand{\R}{\mathbb R}
\newcommand{\C}{\mathcal{C}}
\newcommand{\N}{\mathbb N}

\newcommand{\Leb}{\mathcal{L}}

\def\esup{\mathop\mathrm{ess\,sup\,}}
\newcommand{\Prob}{\mathcal P}

\newcommand{\eps}{\varepsilon}

\newcommand{\ytilde}{\tilde{y}}
\newcommand{\Clambda}{C_{\lambda}}
\newcommand{\Ctildelambda}{\tilde{C}_{\lambda}}


\date{\today}

\author{Camilla Brizzi}
\address{Dipartimento di Matematica e Informatica, Universit\`a di Firenze \\ Viale Morgagni 67/a, 50134 Firenze \\ Italy}
\email{camilla.brizzi@unifi.it}

\author{Luigi De Pascale}
\address{Dipartimento di Matematica e Informatica, Universit\`a di Firenze \\ Viale Morgagni 67/a, 50134 Firenze \\ Italy}
\email{luigi.depascale@unifi.it}
\urladdr{http://web.math.unifi.it/users/depascal/}

\author{Anna Kausamo}
\address{Dipartimento di Matematica e Informatica, Universit\`a di Firenze \\ Viale Morgagni 67/a, 50134 Firenze \\ Italy}
\email{akausamo@gmail.com}

\title[$L^\infty$-optimal transport for a class of quasiconvex cost functions]{$L^\infty$-optimal transport for a class of strictly quasiconvex cost functions}
\begin{document}
\begin{abstract} 
	We consider the $L^\infty$- optimal mass transportation problem 
	\[
	\min_{\Pi(\mu, \nu)} \gamma-\esup  c(x,y),
	\]
	for a new class of costs $c(x,y)$ for which we introduce a tentative notion of twist condition. In particular we study the conditions under which the $\infty$-monotone transport plans are induced by a transportation map. 
	We also state a uniqueness result for  infinitely cyclically monotone Monge minimizers that corresponds to this class of cost functions. We compare the results to previous works. 
\end{abstract}

\maketitle
\section{Introduction}
We consider two Polish spaces  $(X,d_X)$ and  $(Y,d_Y)$.  Let  $c:X\times Y\to\R^+$ be a cost function and $\mu\in \Prob(X), \nu\in \Prob(Y)$; 
here by $\Prob(Z)$ we denote the space of Borel probability measures on a Polish space $Z$. 
Whereas in the classical optimal transport we minimize the integral of the cost function, in the $L^\infty$-optimal transport we consider the problem
\[\min_{\gamma\in\Pi(\mu,\nu)} C_\infty[\gamma]:=\min_{\gamma\in\Pi(\mu,\nu)} \gamma-\esup_{(x,y)\in X\times Y}c(x,y)\tag{P$_\infty$} \]
where 
\begin{equation*}
	\Pi(\mu,\nu):=\left\{\gamma\in \Prob(X \times Y) \, : \, \pi^{1}_{\sharp}\gamma=\mu \ \mbox{and} \ \pi^{2}_{\sharp}\gamma=\nu  \right\}.
\end{equation*}
Above $\pi^{1}$ and $\pi^{2}$ are the projections on the first and the second coordinate of $X\times Y$, respectively.
Minimizers for this problem exist under mild assumptions on $c$ such as lower-semicontinuity. In \cite{ChaDePJuu2008},  Champion, De Pascale, and Juutinen carried out a comprehensive study on the $L^\infty$-transport in the case where the cost of transporting a point $x$ to a point $y$ is given by their distance: $c(x,y)=|x-y|$ on $\R^d$. They studied the problem on compact sets of $\R^d$, $d\ge 1$, and their work was generalized in 2015 by Jylh{\"a} \cite{jylha2015} to Polish spaces with more general costs $c$. Champion, De Pascale, and Juutinen introduced the concept of $\infty$-cyclical monotonicity, which carries the well-known notion of cyclical monotonicity of optimal transportation plans to the $L^\infty$-case.  
In general, optimal $L^\infty$-transportation plans are not as well-behaved as the minimizers of the standard integral Monge-Kantorovich problem. For instance, their restrictions are not always optimal with respect to their marginals. To address this problem, Champion, De Pascale, and Juutinen invoked the concept of \textit{restrictability} which is commonly used in more general $L^\infty$-calculus of variations problems. An optimal $L^\infty$-transportation plan is restrictable if, loosely speaking, its restrictions are also optimal. Moreover, they showed that restrictability and $\infty$-cyclical monotonicity are equivalent. This equivalence also holds for more general cost functions at least if they are continuous, as was proven by Jylh{\"a} in \cite{jylha2015}.

If $X=Y$ and $c=d$, the minimal value $C_\infty(\lambda)$ is  called $\infty$-Wasserstein distance of the measures $\mu$ and $\nu$ and denoted by $W_\infty(\mu,\nu)$. The quantity $W_\infty(\mu,\nu)$ is actually the $p\to\infty$ limit of the $p$-Wasserstein distances $W_p(\mu,\nu)$ of the measures $\mu$ and $\nu$. Also for more general cost functions $c$ we can consider for every $p\ge 1$ the optimal transportation problem
\[ \min_{\gamma\in\Pi(\mu,\nu)}C_p[\gamma]:=\left(\int_{X\times Y} c^{p}(x,y)d\gamma \tag{P$_{p}$}\right)^{\frac{1}{p}}. \]
In the case $c=d$ the $W_\infty$-distance provides a natural control from above for the $W_p$-distances. 
Bouchitt\'e, Jimenez, and Rajesh \cite{BJR2007} and, for even more general costs,  Jylh{\"a} and Rajala \cite{wpwinfty} also established necessary and sufficient conditions for the existence of $W_\infty$-lower bounds for
the integral optimal transportation costs. 

The question on whether there exists a dual formulation for the $L^\infty$-transport, similar to the now-standard Kantorovich duality, remained open until  Barron, Bocea, and Jensen stated and proved a duality theorem in 2017 \cite{BBJ}. The theory was further developed in the $1$-dimensional case by De Pascale and Louet in \cite{dePL}. Unlike in the standard integral optimal transportation, it is not immediate how to use the  $L^ \infty$-duality to prove the existence of deterministic solutions to the Monge-Kantorovich problem. By deterministic solutions we mean minimizers for the problem (P$_\infty$) of the form $\gamma=(Id\times T)_{\sharp}\mu$, where $T$ belongs to the set
\begin{equation*}
	\mathcal{T}(\mu,\nu):=\{ T:X\to Y: T \ \mbox{is a Borel map and } \nu=T_\sharp \mu  \}.
\end{equation*}
This type of solutions are also called Monge solutions, in honor of Gaspard Monge, who is one of the building fathers of the classical theory of optimal transportation. 

In this paper we use a duality-free technique to prove that, under natural assumptions on $\mu$, $\infty$-cyclical monotone transport plans (see Definition \ref{defICM}) --
or, even less, $\infty$-monotone transport plans (these do not need to be optimal) -- are actually induced by a transport map (see Theorem \ref{maintheorem}). 
The technique was introduced by Champion, De Pascale, and Juutinen in \cite{ChaDePJuu2008} where the authors used this property to prove on $\R^d$, in the case where the cost is given by the Euclidean distance, that there exists a deterministic solution for the problem (P$_\infty$)  and, if the target measure $\nu$ is atomic, that restrictable, deterministic solutions are unique. 
Jylh\"{a} in \cite{jylha2015} generalized the existence result to cost functions of the form $c(x,y)=h(y-x)$, where $h:\R^{d}\to \R^{+}$ is strictly quasiconvex in the sense that for all $t\in(0,1)$ and $x,\bar{x}\in \R^{d}$ with $x \neq  \bar{x}$
\[h((1-t)x+t\bar{x})< \max\{h(x),h(\bar{x})\}. \] 
He also extended the uniqueness result of Champion, De Pascale and Juutinen, proving that if $\nu$ contains an atom, say $y_{0}$, and $T$ and $S$ are two optimal transport maps, both corresponding to infinitely cyclically monotone transport plans, then $\mu(T^{-1}(\{y_{0}\})\setminus S^{-1}(\{y_{0}\}))=0$.

The relevant notion in the Monge-Kantorovich problem for integral costs is an invertibility condition on the partial gradient of $c$ \cite{ChaDeP2014,FatFig2010}. This condition is commonly used in dynamical systems, it is called the twist condition and it goes very well along with convexity properties. There is no analogue of the twist condition for $L^\infty$-optimal transport problem. 
In this study we introduce a property which could serve as twist condition in this setting (see Theorem \ref{beforemaintheorem}). And, in addition to giving some examples of costs which have this property, we observe that it is satisfied by some costs studied in previous works. The property we introduce is invariant (as one could expect) by composition of $c$ with a strictly increasing, differentiable function. 

Although we do not change the essence of the proof of uniqueness, we give a new structure to it which, in our opinion, makes the arguments more transparent and may allow a better understanding of the problem. 

The proof of the main results is based on measure theoretic considerations and on the construction of certain specific cones. The same happens in \cite{jylha2015} as well as in previous works in which these techniques are applied. What is different is that the construction of \cite{jylha2015} is based on the translation invariance of the cost while our construction relies on the notion of normal cone to the boundary of a convex set. This points to the possibility of a general construction which may be adapted to different costs.

A different technique, which applies to several costs of the form $c(x,y)=h(x-y)$, to prove the existence of an optimal transport map, was devised by C. Jimenez and F. Santambrogio in \cite{JimSan2012}. In that paper the authors minimize $\int |x-y|^2 d \gamma +\chi (\gamma)$ among the $\gamma \in \Pi (\mu, \nu)$ and 
\[
\chi(\gamma)= \left\lbrace \begin{array}{ll}
	0 & \mbox{if} \ \mbox{supp}(\gamma) \subset \{(x,y) \ | \ c(x,y) \leq M\};\\
	+\infty & \mbox{otherwise}.
\end{array}\right. 
\]
They prove that if $\gamma$ is optimal then it is deterministic. However the optimal map of \cite{JimSan2012} is different from the one we describe here. 

\section{Basic definitions and main theorem}
First we recall some basic results about $L^\infty$-transport, as first established in \cite{ChaDePJuu2008} and then generalized in \cite{jylha2015}. The proofs can be found in these references. 

The first important result is the existence of minimizers for the $L^\infty$-transportation problem. It is a direct consequence of the lower-semicontinuity of the functional $C_\infty$ and the compactness of the set $\Pi(\mu,\nu)$. 
\begin{prp}
	Let $c$ be lower semicontinous. Then the problem ($P_\infty$) has at least one minimizer. 
\end{prp}
In general, there can be a high level of non-uniqueness of minimizers of the functional $C_\infty$: we are only optimizing the ``worst case'' and more locally the situation can be far from optimal. It is often useful to consider a subset of better-behaving minimizers, the so-called \emph{$\infty-c$-cyclically monotone } (ICM) transport plans. The idea is that on the support of such a plan, there is no rearrangement of destinations of a fixed finite set of points that would improve the highest cost coupling of that set. To our purpose it is also useful to have a definition for a plan that can not be improved by the interchange of any two destinations of a given pair of initial points. The formal definitions are the following: 
\begin{dfn}\label{defICM}
	A set $\Gamma\subset X\times Y$ is  $\infty-c$-cyclically monotone (ICM)  if for every finite set of points $\{(x_i,y_i)\}_{i=1}^k\subset \Gamma$ and for every permutation $\sigma$ of the set $\{1,\ldots,k\}$ we have 
	\[\max_{1\le i\le k}c(x_i,y_i)\le \max_{1\le i\le k}c(x_i,y_{\sigma(i)})\,.\]
	We say that a transport plan $\gamma\in \Pi(\mu,\nu)$ is ICM, if it is concentrated on an ICM  set.
\end{dfn}
\begin{dfn}\label{defIM}
	The set  $\Gamma\subset X\times Y$ is $\infty$-monotone (IM) if for any $(x,y),(x',y')\in\Gamma$ we have 
	\[\max\{c(x,y),c(x',y')\}\le \max\{c(x,y'),c(x',y)\}\,.\]
	We say that a transport plan $\gamma\in \Pi(\mu,\nu)$ is IM, if it is concentrated on an IM set. 
\end{dfn}

\begin{rmk}
	Every ICM plan is IM but the reverse is not true in general. Consider, for example, the marginal measures $\mu=\nu$ to be the uniform probability measure on the circle $\mathbb{S}^1$, the cost function $c(x,y)=|x-y|$ (the Euclidean distance), and the set $\Gamma=Graph(R_\theta)$, where $R_\theta$ is the rotation of a vector $v\in \mathbb{S}^1$ by an angle $\theta< \pi / 2 $. Now $\Gamma$ is IM, but since the transport plan supported by $\Gamma$ is not optimal (the optimal cost being $0$, given by the identity map) the set $\Gamma$ cannot be ICM. 
\end{rmk}

$\infty-c$-cyclically monotone transport plans are much better-behaved than arbitrary solutions to the problem ($P_\infty$). For example, they are restrictable in the sense of the following definition.
\begin{dfn}
	A transport plan $\gamma\in\Pi(\mu,\nu)$ is said to be a restrictable minimizer of the problem (P$_\infty$) if it satisfies the following condition: for any $\gamma'\le\gamma$, i.e. $\gamma'(B)\le \gamma(B)$ for every Borel sets $B$, and $\gamma'(X\times Y)>0$ we have
	\begin{equation*}
		C_\infty(\bar{\gamma})=\min\left\{C_{\infty}(\lambda):\lambda\in\Pi(\pi^{1}_\sharp\bar{\gamma},\pi^{2}_\sharp\bar{\gamma})  \right\},
	\end{equation*}
	where $\bar{\gamma}=\frac{\gamma'}{\gamma'(X\times Y)}$.
\end{dfn}
They  exist under relatively mild conditions as the following lemma, taken from \cite{jylha2015}, states. 
\begin{lmm}
	Let $c$ be lower semicontinuous. Let $\gamma_p$ be a solution of
	the problem ($P_p$), $p\geq 1$. Then weak$^\ast$ cluster points
	of $(\gamma_p)_{p\geq 1}$ exist and each of them is a solution of
	the problem ($P_\infty$). If, in addition, $c$ is continuous and
	the minimum of each problem ($P_p$) is finite, then the
	weak$^\ast$ cluster points are ICM (here weak$^\ast$ convergence is tested against continuous and bounded functions).
\end{lmm}

Since in this paper we are interested in the IM plans, namely in proving that they are of the Monge-type, it is useful to state their existence:
\begin{crl}\label{optimalplansareIM}
	If $c$ is continuous and the minimum of the problem (P$_p$) is finite for every $p\ge1$, then the problem (P$_\infty$)  has a solution $\gamma$ that is IM.
\end{crl}

In the following we state some definitions and results which are necessary for the proof of Theorem \ref{beforemaintheorem}. They can also be found in \cite{ChaDePJuu2008}. 
\begin{dfn}
	Let $y\in\R^{d}$, $ r>0$ and let $\gamma\in\Pi(\mu,\nu)$ be a transport plan. We define
	\begin{equation*}
		\gamma^{-1}(B(y,r)):=\pi^{1}\left((\R^{d}\times B(y,r))\cap \supp \gamma\right)
	\end{equation*}
	In other words, $\gamma^{-1}(B(y,r))$ is the set of points whose mass is partially or completely
	transported to $B(y, r)$ by $\gamma$, and $\gamma^{-1}$ corresponds to the inverse of the multimap induced by $\supp \gamma$. Notice also that $\gamma^{-1}(B(y, r))$ is a Borel set.
\end{dfn}
Since this notion is important in the sequel, we recall that when $U$ is a Borel set,
one has
\begin{equation*}
	\lim_{r\to 0^+}\frac{\Leb^{d}\left(U\cap B(x,r)\right)}{\Leb^{d}(B(x,r))}=1,
\end{equation*}
for almost every $x$ in $U$: we shall call such a point $x$ a Lebesgue point of $U$ and we will denote by $\leb (U)$ the Borel set of Lebesgue points. 
\begin{dfn}
	We say that the couple $(x, y)\in \R^{d}\times \R^{d}$ is a $\gamma$-regular point if $x\in\leb\left(\gamma^{-1}(B(y,r))\right)$ for any positive $r$. We denote \begin{equation}\label{regpoints}
		R(\gamma):=\{(x,y)\in \R^{d}\times \R^{d}: (x,y) \ \mbox{is a } \gamma\mbox{-regular point}  \}.
	\end{equation} 
\end{dfn}
\begin{rmk}\label{fullmeasure}
	Notice first that by the closedness of $\supp \gamma$ we have that $R(\gamma)\subset \supp \gamma$. Moreover it is not difficult to show that $R(\gamma)$ is a Borel set. Indeed, if for every fixed $n\in \N$ we consider a countable covering  $\R^{d}\subset\bigcup_{i\in \N}B\left(y_{i,n},2^{-n}\right)$ of balls of radius $2^{-n}$, then
	\begin{equation}\label{Risborel}
		R(\gamma)=\bigcap_{n\ge1}\bigcup_{i\in \N}\left(\leb\left(\gamma^{-1}\left(B\left(y_{i,n},2^{-n}\right)\right)\right)\times B\left(y_{i,n},2^{-n}\right) \right)
	\end{equation}
\end{rmk}

\begin{lmm}\label{gammaconcentrated}
	Let $\gamma\in\Pi(\mu,\nu)$ and assume that $\mu <<\Leb ^{d}$. Then $\gamma$ vanishes outside the set $R(\gamma)$ of $\gamma$-regular points. 
\end{lmm}
\begin{proof}
	Let us denote by $S:=\supp\gamma\setminus R(\gamma)$ the complement of the $\gamma$-regular points, intersected with the support of $\gamma$, {\it i.e.}
	\[S:=\{(x,y)\in \supp \gamma \, : \, x\not\in\leb\left(\gamma^{-1}\left(B(y,r)\right)\right)\ \mbox{for some positive}\ r \}. \]
	One can prove that $S$ is, actually, equal to the set
	\[\bigcup_{n\ge 1}\bigcup_{i\in \N}\left(\left( \gamma^{-1}\left(B(y_{i,n},2^{-n})\right) \setminus \leb \left(\gamma^{-1}\left(B(y_{i,n},2^{-n})\right)\right)  \right)\times B(y_{i,n},2^{-n})\right)\cap\supp\gamma.\]
	In particular we have that
	\[\pi^{1}(S)=\bigcup_{n\ge 1}\bigcup_{i\in \N}\left( \gamma^{-1}\left(B(y_{i,n},\frac{1}{2^n})\right) \setminus \leb \left(\gamma^{-1}\left(B(y_{i,n},\frac{1}{2^n})\right)\right)  \right), \]
	and the set on the right-hand side has Lebesgue measure $0$. Therefore, by absolute continuity of $\mu$ with respect to $\mathcal{L}^d$ we have $\mu (\pi^{1}(S))=0$. Finally 
	\[\gamma(S)\le\gamma (\pi^{1}(S)\times \R^{d})=\mu (\pi^{1}(S))=0. \]
\end{proof}
\noindent\textbf{Notation.} At this point, it is natural to introduce a more refined definition 
\begin{equation*}
	\gamma_* ^{-1}(B(y,r)):=\pi^{1}\left((\R^{d}\times B(y,r))\cap R(\gamma)\right).
\end{equation*}
For future use we also introduce a suitable notation for a cone: let $x_{0},\xi\in\R^{d}$ such that $|\xi|=1$ and $\delta\in[0,2]$ then 
\begin{equation*}\label{cone}
	K(x_{0},\xi,\delta):=\left\{x\in \R^{d}\setminus \{x_0\} \, : \, \frac{x-x_{0}}{|x-x_{0}|}\cdot\xi\ge 1-\delta   \right\} \cup \{x_0\}, 
\end{equation*}
and for a ``truncated'' cone
\begin{equation}\label{compactcone}
	K(x_{0},\xi,\delta,s):=K(x_{0},\xi,\delta) \cap B(x_0,s).
\end{equation}
\begin{prp}\label{densitypointcapcone}
	Let $\mu <<\Leb ^{d}$, let $(x_{0},y_{0})\in R(\gamma)$ and let $r>0$, then for every $\xi \in \partial B(0,1)$, $\delta\in(0,2]$ and $s>0$ it holds:
	\begin{equation*}
		\Leb^{d}\left(\gamma_*^{-1}(B(y_{0},r))\cap K(x_{0},\xi,\delta,s)\right)   >0.
	\end{equation*}
	In general,  if $A$ is a Borel set and $x_0 \in \leb (A)$ it  holds: 
	\[
	\Leb^{d}\left(A \cap \gamma_*^{-1}(B(y_{0},r))\cap K(x_{0},\xi,\delta,s)\right)   >0.
	\]
\end{prp}

\begin{proof}
	By definition of $R(\gamma)$, $x_0$ is a Lebesgue point of $\gamma^{-1}(B(y_0,r))$ which implies $\Leb^d\left(\gamma^{-1}(B(y_0,r))\cap K(x_0,\xi, \delta,s)\right)>0$ for every $\xi\in \partial B(0,1)$, $\delta\in (0,2]$, $s>0$. Thus there exists \[(x,y)\in (K(x_0,\xi, \delta,s)\times B(y_0,r))\cap \supp\gamma\] such that $x\neq x_0$. We take $\rho>0$ small enough so that $B(x,\rho)\subset K(x_0,\xi,\delta,s)$ and $B(y,\rho)\subset B(y_0,r)$. We know that $\gamma(B(x,\rho)\times B(y,\rho))>0$ because $(x,y)\in\supp \gamma$, and since $\gamma$ is concentrated in $R(\gamma)$ we have:
	\begin{align*}
		0&<\gamma((B(x,\rho)\times B(y,\rho))\cap R(\gamma))\\
		&\le\gamma((K(x_0,\xi,\delta,s)\times B(y_0,r))\cap R(\gamma))\\
		&\le\gamma(K(x_0,\xi,\delta,s)\cap\gamma_*^{-1}(B(y_0,r))\times\R^d)\\
		&=\mu(K(x_0,\xi,\delta,s)\cap\gamma_*^{-1}(B(y_0,r)).
	\end{align*}
	Which in turn implies $\Leb^d(K(x_0,\xi,\delta,s)\cap\gamma_*^{-1}(B(y_0,r))>0$ by the fact that $\mu<<\Leb^d$.
\end{proof}
\begin{crl}\label{positivemeasure}
	Let $(x_{0},y_{0})\in R(\gamma)$ and let $C$ be a convex set with $\interior C \neq \emptyset $ such that $x_{0}\in C$. Then for every $r>0$ 
	\begin{equation*}
		\Leb^{d}(\gamma_*^{-1}(B(y_{0},r)\cap \interior C ))>0.
	\end{equation*}
	Moreover, if $A$ is a Borel set and $x_0\in\leb(A)$
	\begin{equation*}
		\Leb^{d}(A\cap \gamma_*^{-1}(B(y_{0},r)\cap \interior C ))>0.
	\end{equation*}
\end{crl}
\begin{proof}
	We observe that by convexity of $C$ there exists a cone $K\left(x_{0},\xi,\delta\right)$, for some $\xi\in\partial B(0,1)$ and $\delta \in (0,1)$, such that for $s$ sufficiently small 
	\begin{equation*}
		K\left(x_{0},\xi,\delta,s \right) \subset \interior C\cup \{x_0 \} .
	\end{equation*}
	By the monotonicity of the Lebesgue measure, the claim now follows from Proposition \ref{densitypointcapcone}.
\end{proof}

Before the first main result of this paper we introduce quasiconvex and strictly quasiconvex functions. We also present a technical lemma on convex sets that will be useful in the sequel.
\begin{dfn} We say that a function $h:\R^d \to \R$ is quasiconvex if for all $x,y\in\R^d$ and for all $t \in [0,1]$ we have 
	\[h((1-t)x+ty)\leq \max\{h(x),h(y)\}.\]
	We say that $h$ is strictly quasiconvex if for all $x\neq y$ and for all $t \in (0,1)$  
	\[h((1-t)x+ty)< \max\{h(x),h(y)\}.\]
\end{dfn}

\begin{rmk}\label{levelsets} Immediate properties are the following
	
	\begin{enumerate}
		\item $h$ is quasiconvex if and only if for all $\lambda \in \R$ the sublevel set $C_\lambda=\{x \ : \ h(x) \leq \lambda\}$ is convex (possibly empty);
		\item if $h$ is continuous and strictly quasiconvex, then the only sublevel sets (possibly empty) with empty interior are the sets $C_m$ with $m\le\inf h$; 
		\item if $h$ is continuous and strictly quasiconvex, then the level set $\{x \ : \ h(x)=\lambda\}$ has always empty interior. 
	\end{enumerate}	
\end{rmk}
\begin{lmm}\label{existencecone}
	Let $B,C\subset\R^{d}$ be two closed convex sets with nonempty interiors, differentiable boundaries, and such that  $B\cap C\neq\emptyset$.
	Let $x\in \partial B\cap\partial C$ be such that $n_{B}(x)\neq n_{C}(x)$, where $n_{B}(x)$ and $n_{C}(x)$ are the unit outer normals of $B$ and $C$ at $x$, respectively.
	Then there exists a point $a\in \interior {B}\cap\{w: (w-x)\cdot n_{C}(x)>0  \}$ and $\delta,s >0$ such that the intersection of the cone of direction $a-x$ and amplitude $\delta$ with the ball centered at $x$ and of radius $s$ is all contained in $\interior B$, that is 
	\begin{equation}
		\label{coneinside}K\left(x,\frac{a-x}{|a-x|},\delta, s\right)\subset \interior B\cup\{x\}.
	\end{equation}
	Moreover for the reverse cone we have
	\begin{equation}
		\label{reversecone}
		K\left(x,-\frac{a-x}{|a-x|},\delta, s \right) \subset \interior C \cup \{x\}.
	\end{equation} 
\end{lmm}
\begin{proof}
	By the differentiability of the boundaries of sets $B$ and $C$ we have that $n_{B}(x)$ and $n_{C}(x)$ are well-defined elements of $\R^d$. In particular, by convexity $n_{B}(x)$ is the only unit vector such that 
	\begin{equation*}
		(b-x)\cdot n_{B}(x)\le 0, \quad \mbox{for every }b\in B.
	\end{equation*}
	Hence there exists $a\in \interior B$ such that 
	\begin{equation}\label{posscalarprod}
		(a-x)\cdot n_{C}(x)>0.
	\end{equation}
	By choosing a suitable $s<|a-x|$ and $\delta$ close enough to $0$, by the continuity of the scalar product and the convexity of $B$ we can fix the cone $K\left(x,\frac{a-x}{|a-x|},\delta,s \right)$ that satisfies the first claim \eqref{coneinside}.\\
	Moreover, by Condition \eqref{posscalarprod} we have $-\frac{a-x}{|a-x|}\cdot n_C (x)<0$. 
	Thus for a suitable choice of $\tau>0$, we have 
	\footnote{There exist $\eps>0$ and a concave function $\phi:B^{d-1}(0,\rho)\to\R$ such that $x=(\mathbf{0},\phi(\mathbf{0}))$ and such that $\partial C$ is the graph of $\phi$ in a sufficiently small neighborhood of $x$. In this setting $n_{C}(x)=\dfrac{1}{\sqrt{1+|\nabla\phi (\mathbf{0})|^2}}(-\nabla \phi (\mathbf{0}),1)$. We use the notation  $(\mathbf{v},v_d)=v=-\frac{a-x}{|a-x|}$, where $\mathbf{v}\in \R^{d-1}$ and $v_d\in\R$. The condition $v\cdot n_C(x)<0$ becomes $\mathbf{v}\cdot \nabla \phi (\mathbf{0})>v_d$. Proving that $x+\tau v\in \interior C$, for some $\tau>0$, can be done by showing that $\phi(\tau \mathbf{v})> \phi (\mathbf{0})+\tau v_d$. The latter inequality follows directly by the fact that $\lim_{t\to 0}\frac{\phi(t\mathbf{v})-\phi (\mathbf{0})}{t}> v_d$.  } $x-\frac{a-x}{|a-x|}\tau\in \interior C$. Arguing as before, there exists $s'<\left|x-\frac{a-x}{|a-x|}\tau-x\right|=\tau$ and $\delta'$ close enough to $0$, such that $K\left(x,-\frac{a-x}{|a-x|},\delta',s'\right)$ satisfies the \eqref{reversecone}. 
	Therefore, possibly substituting  $s$ with $\min\{s,s'\}$ and $\delta$ with $\min\{\delta,\delta'\}$, \eqref{coneinside} and \eqref{reversecone} hold simultaneously.
\end{proof}

\noindent\textbf{Notation.}
Given $y\in \R^d$ and $\lambda \in \R$ we denote, whenever it is possible, by 
\[n_{c(\cdot, y)}(x)\]
the unit outer normal to the sublevel set $C_\lambda=\{z \ : \ c(z,y) \leq \lambda\}$ at the 
point $x \in \partial C_\lambda$. 

\begin{thm}\label{beforemaintheorem}
	Let $\mu,\nu$ be two Borel probability measures on $\R^{d}$ with compact supports and $\mu<<\Leb^{d}$. Let $c:\R^{d}\times\R^{d}\to\R^{+}$ be a continuous function satisfying the following properties:
	\begin{enumerate}[(i)]
		\item\label{costass1} For all $x$ there exists a unique $y\in \R^d$ such that $c(x,y)=0$, and for every $y\in\R^d$ we have \footnote{By the strict quasiconvexity of $c(\cdot,y)$ (Assumption (ii)), if there exists a point $x$ such that $c(x,y)=\inf_{z\in\R^d}c(z,y)$, then such a point is also unique.} $\inf_{x\in\R^d}c(x,y)=0$;
		\item\label{costass2} $c(\cdot,y)$ is strictly quasiconvex for every $y$, i.e. 
		$$c((1-t)x+t\bar{x},y)<\max\{c(x,y),c(\bar{x},y)\}$$ for all $t\in(0,1)$ 
		and for every $x,\bar{x}\in\R^{d}$, $x\neq\bar{x}$;
		\item\label{costass3}  for all $\lambda >0$ and $y\in\R^d$ the convex set $C_\lambda= \{z \; : \; c(z,y) \leq \lambda\}$ has 
		differentiable $(\C^1)$ boundary;
		\item\label{costass4}
		$c$ satisfies a ``twist kind'' condition, that is: for every $x,y,\tilde{y}\in\R^{d}$ and for all $\lambda>0$
		\begin{equation*}
			\begin{cases}
				c(x,y)=c(x,\tilde{y})=\lambda \\
				n_{c(\cdot, y)}(x)=n_{c(\cdot, \tilde y)}(x)
			\end{cases}\implies y=\tilde{y}.
		\end{equation*}
	\end{enumerate}
	Let $\gamma\in\Pi(\mu,\nu)$ and $(x,\tilde{y})\in R({\gamma})$. Then for every $(x,y)\in \R^d\times\R^d$ with $y\neq\tilde y$, there exists $(x',y')\in R({\gamma})$ such that 
	\begin{equation}\label{improvement}
		\max\{c(x',y), \, c(x,y')\}<\max\{c(x,y), \, c(x',y')\}.
	\end{equation}
	Moreover, if $x\in\leb(A)$, for some Borel set $A$, then $(x',y')$ above can be taken with $x'\in A$.
\end{thm}
\begin{rmk}
	We call condition ($\ref{costass4}$) above ``twist kind'' because whenever $c$ is differentiable with respect to $x$, if  $\nabla_{x}c(x,\tilde{y})\neq0$ and $\nabla_{x}c(x,y)\neq 0$, then the condition becomes
	\begin{equation*}
		\begin{cases}
			c(x,y)=c(x,\tilde{y})\\
			\frac{\nabla_{x}c(x,y)}{|\nabla_{x}c(x,y)|}=\frac{\nabla_{x}c(x,\tilde{y})}{|\nabla_{x}c(x,\tilde{y})|}
		\end{cases}\implies y=\tilde{y}.
	\end{equation*}
	This looks like the twist condition in the classical integral optimal transport.
\end{rmk}
\begin{proof}
	\underline{\textbf{Case 1}}: $c(x,y)>c(x,\ytilde)$. \\
	First of all we fix $\lambda=c(x,y)$ and we consider $C_\lambda=\{z \ : \ c(z,y)\le \lambda\}$. Notice that $\interior {C_{\lambda}}$ is nonempty by part (2) of Remark \ref{levelsets}. Indeed $c(x,y)=\lambda >0=\inf_{w\in\R^d} c(w,\tilde y)$, by Assumption (i). The continuity of the cost function allows us to fix $r>0$ such that 
	\begin{equation}\label{bycontinuity}
		c(x,y')<c(x,y)~~~\text{for all } y'\in  B(\tilde{y},r)\,.
	\end{equation} 
	Moreover, since $(x,\ytilde)\in R(\gamma)$, thanks to Corollary \ref{positivemeasure}, we know that there exists $(x',y') \in R(\gamma)$ with $x' \in \interior {C_{\lambda}=\{z \ : \ c(z,y)< \lambda\}}$ (so that $x\not=x'$) and  $y'\in B(\tilde{y},r)$ that is
	\[(x',y')\in \left(\interior {C_{\lambda}} \times B(\tilde y,r)\right)\cap R(\gamma)\,.\]
	Notice that such a $y'$ exists by the definition of $ \gamma_*^{-1}(B(\tilde y,r))$ and it could happen that $y=y'$. 
	By Condition (\ref{bycontinuity}) we have  
	\[c(x,y')<c(x,y)\,.\] 
	We also have that 
	\[c(x',y)< c(x,y)\, , \]
	because $x'\in \interior {C_{\lambda}}$.
	These two inequalities imply that 
	\begin{equation*}
		\max\{c(x',y),c(x,y')\}< c(x,y) \le\max\{c(x,y),c(x',y')\},
	\end{equation*}
	concluding the proof for the Case 1.\\
	\underline{Case 1 bis:} If $x\in \leb(A)$, for a Borel set $A$, then by the second part of Corollary \ref{positivemeasure}, we can find $(x',y')\in R(\gamma)$, such that $x'\in \interior C_ \lambda\cap A$ and $y\in B(\tilde{y},r)$. As before, such $(x',y')$ satisfies the \eqref{improvement}.\\
	\underline{\textbf{Case 2}}: $c(x,y)<c(x,\tilde y)$. \\
	We follow a path similar to that of \cite{jylha2015} but with a different construction of cones.\\
	Fix $\lambda=c(x,\tilde{y})$, so that $x \in \partial\Ctildelambda$, where $\Ctildelambda=\{z \ : \ c(z,\tilde y)\le\lambda  \} $. Again we denote by $n_{c(\cdot, \tilde y)} (x)$  the unit outer normal to the set $\tilde C_\lambda$ at $x$. \\
	Let $s$ be a small, positive real number such that $\tilde a:= x-sn_{c(\cdot, \tilde y)} (x)\in \interior \tilde C_\lambda$, which is not empty thanks to  $(2)$ of Remark \ref{levelsets}.  
	We fix \[\varepsilon= \frac{1}{4}\min\left\{{c(x,\tilde y)- c(\tilde a, \tilde y), c(x, \tilde y) -c(x,y)}\right\}.\] Thanks to the continuity of $c$ it is possible to find a positive radius 
	$r$ such that $B(\tilde{a},r)\subset\tilde{C}_\lambda$, for every $(\tilde z,y') \in B(\tilde a, r) \times B(\tilde y, r)$ we have 
	\begin{equation}
		\label{bycontinuity1}
		c(\tilde z,y') < c(\tilde a , \tilde y) + \varepsilon
	\end{equation}
	and for every $y' \in B(\tilde y, r) $ it holds 
	\begin{equation}\label{bycontinuity2}
		c(x, \tilde y) - \varepsilon < c(x,y').
	\end{equation}
	Consider the set $\tilde K= (1-t)x + t B(\tilde a, r)$ with\footnote{We use the notation $\tilde K$ because also this set is a cone, even if slightly different from the cones  defined by \eqref{compactcone}.}  $t \in (0,1)$. Thanks to the convexity, $\tilde K$ is contained in $\tilde{C}_\lambda$. We claim that
	\[\tilde K\subset \{z \ : \ c(z,y') < c(x,y')\} \quad \text{for every} \ y' \in B(\tilde y,r). \] 
Let $z\in\tilde K$. Then $z=(1-t)x+t\tilde z$, for some $\tilde z\in B(\tilde{a},r)$ and some $t\in(0,1)$. We observe that $c(\tilde z,y')< c(x,y')$. Indeed
		\begin{equation}\label{Ktildeinclusion}
			c(\tilde z,y')<c(\tilde{a},\tilde{y})+\eps\le \frac{1}{4}c(x,\tilde{y})+\frac{3}{4}c(\tilde a,\tilde y)<c(x,\tilde y)-3\eps<c(x,y'),
		\end{equation}
		where the first inequality is due to Condition \eqref{bycontinuity1}, the second and the third to the definition of $\eps$ and the last one to \eqref{bycontinuity2}.
		We conclude by the strict quasiconvexity of $c(\cdot, y')$, which implies $c(z,y')<\max\{c(x,y'), \, c(\tilde z,y')\}=c(x,y')$.\\
	On the other hand, if we consider the reflected cone $K= (1-t)x+tB(\tilde a, r)$ with $t \in (-1,0)$, again by the strict quasiconvexity of $c(\cdot, y')$ we show that
	\begin{equation*}
		K \subset \{z \ : \ c(z,y') > c(x,y')\}\quad \text{for every } y' \in B(\tilde y,r).
	\end{equation*}
	Indeed, let $z\in K$, then $z=(1-t)x+t\tilde z$, for some $\tilde z \in B(\tilde a,r)$ and $t\in (-1,0)$. So $x= sz+(1-s)\tilde z$, with $s=\frac{1}{(1-t)}$. Thus $c(x,y')<\max\{c(z,y'),c(\tilde z,y')\}=c(z,y')$, where the last equality is due to \eqref{Ktildeinclusion},
	and therefore for every $x' \in K$ and $y' \in B(\tilde y, r) $ it holds
	\begin{equation*}
	c(x,y')< c(x',y').
	\end{equation*}
	Let now $\rho>0$ such that $c(x',y)<c(x,y)+\eps$ for every $x'\in K\cap B(x,\rho)$. Then for every $x'\in K\cap B(x,\rho)$ and $y'\in B(\tilde y, r)$ it also holds:
	\[c(x',y)<c(x,y)+\eps<\frac{1}{4}c(x,\tilde y)+\frac{3}{4}c(x,y)\le c(x,\tilde y)- 3\eps< c(x,y')<c(x',y'). \]
	The proof of Case 2 is concluded thanks to Proposition \ref{densitypointcapcone}, for which it is possible to choose $x'$ in the set 
	\[\left(B(x,\rho)\cap K \cap \gamma_*^{-1}(B(\tilde y,r) )\right)\] and therefore $(x',y')$ in the set  
	\[\left(\left( K\cap B(x,\rho)\right)\times B(\tilde y, r) \right)\cap R(\gamma).\] 
	\underline{Case 2 bis:} If $x\in \leb(A)$ for some Borel set $A$, by the last part of Proposition \ref{densitypointcapcone}, we can choose $(x',y')$ satisfying the \eqref{improvement} in the set 
	\[\left(\left(A\cap K\cap B(x,\rho)\right)\times B(\tilde y, r) \right)\cap R( \gamma).\] 
	\underline{\textbf{Case 3}}: $c(x,y)=c(x,\ytilde)=\lambda>0$ (if $c(x,y)=c(x,\ytilde)=0$ then $y=\tilde y$ by Assumption (i)) and $ n_{c(\cdot,y)} (x) \neq  n_{c(\cdot,\tilde y)} (x)$. \\ In Figure \ref{case3maintheorem} we provide a visual description of the proof of this case, where we have denoted the vector $n_{c(\cdot,\tilde y)}(x)$ by $\overrightarrow{n}$.
	
	\begin{figure}[h]
		\begin{tikzpicture}[scale=2.5]
			\coordinate (O) at (0,0);
			\coordinate (b) at (3,0);
			\coordinate (c) at (3,1);
			\coordinate (d) at (1.7,1.6);
			\coordinate (e) at (0,1);
			
			\path[name path = convex1, draw,use Hobby shortcut, thick, closed=true ]
			(O)..(b)..(c)..(d)..(e) ;
			\draw (e) node[anchor=east]{$C_{\lambda}$};
			
			\coordinate (f) at (3,-1);
			\coordinate (g) at (6,0);
			\coordinate (h) at (5,2);
			\coordinate (i) at (3,1);

			\path[name path = convex2, draw,use Hobby shortcut,thick, closed=true]
			(f)..(g)..(h)..(i);
			
			\draw (h) node[anchor=south west]{$\tilde{C}_{\lambda}$};
			\node [thick, label=90:$x$] at (i) {};
			\draw (i) node{$\bullet$};
			
			\path[name path = aux1] (i) circle [radius = 1bp];	
			\path[name intersections = {of = convex2 and aux1}] ; 
			
			\draw [-latex, thick, name = normal] (i)--node[above=0.2pt]{$\overrightarrow{n}$} 
			($(i)!0.75cm!90:($(intersection-2)!.75cm!(intersection-1)$)$) ; 

			\coordinate (a') at (1.7,1.4);
			\coordinate (b') at (1.7,0.6);
			\filldraw [white, name path = axis1] (a') --
			($(a')!0.75cm!270:(i)$);
			\filldraw  [white, name path = axis2](b') --
			($(b')!0.75cm!90:(i)$);
			\path[name intersections = {of = axis1 and axis2, by=p}];
			
			\tkzDrawCircle[name path = biggcircle, color=white, scale = 1.05](i,b');
			\tkzDrawCircle[name path = bigcircle, thin](i,b');
			\tkzDrawCircle[name path = smallcircle, thin, scale = 0.5](i,b');
			\coordinate (B) at (2.7,2.3); 
			\draw (B) node[anchor=south]{$B(x,s)$}; 
			\coordinate (B') at (3.5,0.3);
			\draw (B') node{$B(x,\frac{s}{2})$}; 
			
			\draw [name path =coneedge1, thick] ($(i)!1.6cm!(a')$)--($(i)!1.6cm!180:(a')$); 
			\draw [name path =coneedge2, thick] ($(i)!1.6cm!(b')$)--($(i)!1.6cm!180:(b')$)node[anchor=north] {$\tilde{K}$}; 
			\draw [name path= coneaxis, dashed] ($(i)!1.6cm!(p)$)--($(i)!1.6cm!180:(p)$); 
			\path[name intersections = {of = biggcircle and coneaxis, by=a}]; 
			\draw (a') node[anchor=south west] {$K$}; 
			
			\node [label={[label distance=0.02cm]170:$a$}] at (a) {}; 
			\draw (a) node{$\bullet$};
			\coordinate (x') at (1.9,1.2); 
			\draw (x') node{$\bullet$};
			\node [label=180:$x'$] at (x') {};
			\draw [name path =line] ($(i)!1.6cm!(x')$)--($(i)!1.6cm!180:(x')$);
			\coordinate (z) at ($(i)!1cm!180:(x')$);
			\draw (z) node{$\bullet$};
			\node [label=10:$z$] at (z) {};
			\coordinate (y) at (1,0);
			\draw (y) node{$\bullet$};
			\node [label=180:${y}$] at (y) {};
			\coordinate (ytilde) at (5.5,0.5);
			\draw  (ytilde) node{$\bullet$};
			\node [label=180:$\tilde{y}$] at (ytilde) {};
			\draw[name path = circleytilde] (ytilde) circle [radius = 0.3cm];
			\coordinate (q) at (5.8,0.5);
			\draw [dashed ](ytilde)-- node[below=1pt] {$r$}(q);
			\coordinate (y') at (5.6,0.6);
			\draw (y') node{$\bullet$};
			\node [label=0:$y'$] at (y') {};
		\end{tikzpicture}
		\caption{Case 2 - Theorem \ref{beforemaintheorem}} \label{case3maintheorem}
	\end{figure}
	\noindent Since $x\in\partial\Clambda\cap\partial\Ctildelambda$ and since by part (2) of Remark \ref{levelsets} these sublevel sets have nonempty interiors, the assumptions of Lemma \ref{existencecone} above are satisfied. Therefore we can construct two ``truncated'' cones, one inside $\interior{\Clambda}\cup\{x\}$ and the ``reverse'' one contained in $\interior\Ctildelambda$. More precisely there exist $a \in \interior{\Clambda}$ and $\delta, s>0$ such that 
	\begin{align*}
		&K\left(x,\frac{a-x}{|a-x|},\delta,s \right)\subset \interior C_\lambda\cup\{x\} \ \mbox{and} \ \\
		&K\left(x,-\frac{a-x}{|a-x|},\delta,s \right)\subset{\interior\Ctildelambda}\cup\{x\}.
	\end{align*}
	In order to simplify the notation, let us denote by
	\begin{equation*}
		K:=K\left(x,\frac{a-x}{|a-x|},\delta,s \right) \quad \mbox{and} \quad \tilde{K}:=K \left(x,-\frac{a-x}{|a-x|},\delta,s \right).
	\end{equation*}
	Since $c(x,\ytilde)>c(z,\ytilde)$ for every $z\in \interior \Ctildelambda$, if we take $z$ in a set well-contained in $\interior \Ctildelambda$, for instance the portion of annulus $\tilde{K}\setminus B\left(x,\frac{s}{2}\right)$, by the continuity of the cost there exists $r>0$ such that 
	\begin{equation}\label{inequality}
		c(x,y')>c(z,y'), \quad \mbox{for all }y'\in B(\ytilde, r), \ \mbox{for all }z\in \tilde{K}\setminus B\left(x,\frac{s}{2}\right).
	\end{equation}
	Let $(x',y')\in \left( K \times B(\ytilde,r)\right)\cap R(\gamma)$ with $x'\neq x$, which exists by Proposition \ref{densitypointcapcone} (and definition of $\gamma_*^{-1}(\tilde y,r)$).
	Consider $z\in \tilde{K}\setminus B\left(x,\frac{s}{2}\right)$ on the line passing from $x'$ and $x$ (the order of the points on the line being $x'$, $x$, $z$), by the strict quasiconvexity of $c$ with respect to the first variable we have
	\begin{equation*}
		c(x,y')<\max\{c(x',y'),c(z,y') \}.
	\end{equation*}
	By Condition \eqref{inequality}, we infer that $c(x,y')<c(x',y')$. \\
	Finally, since $x' \in\interior{\Clambda}$, we have $c(x',y)<c(x,y)$ (again we use that, by the continuity and the strict quasiconvexity of $c$, $\partial C_\lambda =\{z \ : c(z,y)=\lambda\}$). Hence Condition \eqref{improvement} holds.\\
	\underline{Case 3 bis:} Also in this case thanks to the second part of Proposition \ref{densitypointcapcone}, we can choose our $(x',y')\in ((A\cap K)\times B(\tilde y, r))\cap R( \gamma)$ if $x\in \leb(A)$, for some Borel set $A$.\\
	\underline{\textbf{Case 4}}: $c(x,y)=c(x,\ytilde)=\lambda>0$ and 
	\[ n_{c(\cdot,y)} (x) = n_{c(\cdot,\tilde y)} (x).\]
	If this case occurs, by the assumption \eqref{costass4} (the ``twist kind'' condition) we know that $y=\ytilde$ against our assumption.
\end{proof}
We are now ready to state and prove the main result of this section, which follows directly by Theorem \ref{beforemaintheorem}.
\begin{thm}\label{maintheorem}
Let $\mu,\nu$ be two Borel probability measures on $\R^{d}$ with compact supports and $\mu<<\Leb^{d}$ and let $c:\R^{d}\times\R^{d}\to\R^{+}$ be a continuous function satisfying the assumptions \eqref{costass1}-\eqref{costass4} of Theorem \ref{beforemaintheorem}.
	If $\gamma\in\Pi(\mu,\nu)$ is $\infty$-monotone, then $\gamma$ vanishes outside the graph of a Borel map $T\in\mathcal{T}(\mu,\nu)$. 
\end{thm}
\begin{proof}
	By Lemma \ref{gammaconcentrated} we know that $\gamma$ is concentrated on the set of $\gamma$-regular points $R(\gamma)$ and that the set $R(\gamma)$ is a Borel set. Therefore, it suffices to show that  the set $R(\gamma)$ is the graph of a function $T\in\mathcal{T}(\mu,\nu)$. More precisely it is enough to prove that the set $R(\gamma)$ is contained in the graph of some function $T:\R^{d}\to \R^{d}$. Indeed, Theorem 2.3 of \cite{Ambrosio} ensures that $T$ is a Borel map and that $\gamma=(Id\times T)_{\sharp}\mu$. 
	Let us assume, by contradiction, that there exist $(x,y),(x,\tilde{y})\in R(\gamma)$ with $y\ne \tilde y$. Then we can apply Theorem \ref{beforemaintheorem} thanks to which there exists $(x',y')\in R(\gamma) $ such that 
	\begin{equation*}
		\max\{c(x',y), \, c(x,y')\}<\max\{c(x,y), \, c(x',y')\},
	\end{equation*}
	contradicting the $\infty$-monotonicity of $\gamma$.
\end{proof}

\section{On the uniqueness of ICM optimal transport plans}
The next lemma is slightly less general than its equivalent in \cite{jylha2015} but we believe it makes more transparent the proof of uniqueness that will follow. 

\begin{lmm}\label{doublepoints} Let $\mu$ and $\nu$ be two probability measures on $\R^{d}$, with compact supports and  $\mu<<\Leb^{d}$. Assume that the function $c:\R^{d}\times \R^{d}\to \R^{+}$ satisfies the assumptions of Theorem \ref{beforemaintheorem}. Let $T, \tilde T \in \mathcal{T}(\mu, \nu)$ be two transport maps associated to two $\infty$-monotone transport plans $\gamma$ and $\tilde \gamma$, i.e., $\gamma=(Id\times T)_{\sharp}\mu$ and $\tilde \gamma=(Id\times \tilde T)_{\sharp}\mu $.
	Let $A \subset \R^d$ be a Borel set such that for all $x\in A$, one has $T(x) \neq \tilde T (x)$. Define 
	\begin{multline*}
		\mathcal{R} (A)=\{ z \in \R^d \ : \ \exists \, x \in A \ \mbox{s.t.} \\
		\max \{c(x,T(z)), c(z,\tilde T(x))\} < \max \{c(x,\tilde T(x)), c(z,T(z)) \}\}.
	\end{multline*}
	Then $\mu(A \setminus \mathcal{R} (A))=0$.
\end{lmm}
\begin{rmk}
	The definition of $\mathcal{R}(A)$ is not symmetric with respect to the interchange of roles of $T$ and $\tilde T$, in the sense that 
	\begin{multline*}
		\mathcal{\tilde R} (A)=  \{ z \in \R^d \ : \ \exists \, x \in A \ \mbox{s.t.} \\
		\max \{c(x,\tilde T(z)), c(z,T(x))\} < \max \{c(x,T(x)), c(z,\tilde T(z)) \}\}.
	\end{multline*}
	is not necessarily equal to $\mathcal{R}(A)$. However, the same result can also be proven for $\mathcal{\tilde R}(A)$.
\end{rmk}
\begin{proof} First of all we observe that \[\begin{split}
		\mu(A\setminus\mathcal{R}(A))&=\mu\left(\pi^1(R(\gamma))\cap\pi^1(R(\tilde\gamma))\cap A \setminus \mathcal{R}(A)\right)\\
		&=\mu\left(\pi^1(R(\gamma))\cap\pi^1(R(\tilde\gamma))\cap \leb(A) \setminus \mathcal{R}(A)\right),\end{split} \]
	since $\gamma$ and $\tilde \gamma$ are respectively concentrated on $R(\gamma)$ and $R(\tilde\gamma)$ and $\mu<<\Leb^d$. We will prove that \[\pi^1(R(\gamma))\cap\pi^1(R(\tilde\gamma))\cap \leb(A)\setminus \mathcal{R}(A)=\emptyset.\] 
	Let us assume by contradiction that there exists $x\in\pi^1(R(\gamma))\cap\pi^1(R(\tilde\gamma))\cap \leb(A)\setminus \mathcal{R}(A)$ and consider $(x,T(x))\in R(\gamma)$ and $(x,\tilde T(x))\in R(\tilde\gamma)$. By definition of $A$, we have $T(x)\neq\tilde T(x)$.
	By the final sentence of Theorem \ref{beforemaintheorem} applied to $R(\tilde\gamma)$, we can find $(x',y')=(x',\tilde T(x'))\in R(\tilde \gamma)$ such that $x'\in A$ and
	\begin{equation}\label{ineq}
		\max\{c(x,\tilde T(x')),c(x',T(x))\}<\max\{c(x, T(x)), c(x',\tilde T(x'))\}.
	\end{equation}
	We conclude noticing that the inequality \eqref{ineq} implies that $x\in\mathcal{R}(A)$, contradicting the fact that $x\in A\setminus \mathcal{R}(A)$.
\end{proof}
\begin{thm}\label{uniqueness}
	Let $\mu$ and $\nu$ be two probability measures on $\R^{d}$, with compact support and  $\mu<<\Leb^{d}$. Assume that the function $c:\R^{d}\times \R^{d}\to \R^{+}$ satisfies the assumptions \eqref{costass1}-\eqref{costass4} of Theorem \ref{beforemaintheorem} and that $\nu(\{y_{0}\})>0$ for some $y_{0}\in\R^{d}$. Let $T$ and $\tilde T$ be optimal transport maps corresponding respectively to $\infty-c$-cyclically monotone transport plans $\gamma$ and $\tilde\gamma$. Then
	\begin{equation*}
		\mu\left(T^{-1}(\{y_{0}\})\setminus \tilde T^{-1}(\{y_{0}\})\right)=0.
	\end{equation*}
\end{thm}
\begin{proof} Assume, by contradiction, that
	\begin{equation*}
		\mu\left(T^{-1}(\{y_{0}\})\setminus \tilde T^{-1}(\{y_{0}\})\right)>0.
	\end{equation*}
	We first restrict to the set of the full $\mu$-measure where both the graphs of $T$ and $\tilde T$ are  $\infty-c$-cyclically monotone (by Theorem \ref{maintheorem} we know that we may consider the set $\pi^1(R(\gamma))\cap \pi^1(R(\tilde\gamma))$).
	We may also assume, restricting our attention to the set of interest, that 
	\begin{equation}\label{differentimages}
		T(x)\ne \tilde T(x)~~~\text{ for all }x\in \pi^1(R(\gamma))\cap \pi^1(R(\tilde\gamma)).
	\end{equation}
	In this proof we will apply Lemma \ref{doublepoints} repeatedly to different subsets $A$ of $\supp\mu$, this is always possible by Assumption (\ref{differentimages}). 

		We define $B_0=\{y_0\}$, $A_1=T^{-1}(B_0)$, $A_1'=\mathcal{R}(A_1)$, and $B_1=\tilde T(A_1')$. We continue recursively: assuming that for all $j\in\{1,\ldots,k-1\}$ the sets $A_j,A_j'$, and $B_j$ have already been defined we set
		\begin{align*}
			&A_k=T^{-1}(B_{k-1})\,,~~~A_k'={\mathcal{R}}(A_k)\,,~~~\text{and}\\
			&B_k=\tilde T(A_k').
		\end{align*}
		We observe that by construction $\mu(A_{j})>0$ and, by Lemma \ref{doublepoints}, also $\mu(A_j')>0$ for every $j$. We continue by defining
		\[P_1=\bigcup_{k\ge 1}A_k~~~\text{and}~~~P_2=\bigcup_{k\ge 1} B_k.\]
		We prove the following intermediate claim:\\\\
		\underline{\textbf{Claim}} $y_0\in P_2$. \\
		\textit{Proof. }Let us first show that 
		\begin{equation}\label{eq:1}
			\mu(P_1)=\nu(P_2).
		\end{equation}
		We prove that $\nu(P_2)\le \mu(P_1)$. Right from the definitions of the sets $P_1$, $P_2$, $A_k$, and $B_k$ we see that
		\[T^{-1}(B_0\cup P_2)=P_1.\]
		Since $\nu=T_\sharp\mu$, we get
		\begin{equation}\label{eq:2}
			\nu(P_2)\le\nu(B_0\cup P_2)=\mu(T^{-1}(B_0\cup P_2))=\mu(P_1).
		\end{equation}
		In order to prove the opposite inequality we first observe that, by using Lemma \ref{doublepoints} with $A=A_k$, we have
		\begin{equation}\label{eq:extra}
			\mu(A_k\setminus\mathcal{R}(A_k))=0~~~\text{for all }k\ge 1;
		\end{equation}
		and we recall that in this construction $A_k'=\mathcal{R}(A_k)$. 
		Now 
		\begin{align*}
			\mu(P_1)&=\mu\left(\bigcup_{k\ge 1}A_k\right)\stackrel{a)} \le\mu\left(\bigcup_{k\ge 1}A_k'\right)\\
			&\le\mu\left(\bigcup_{k\ge 1}\tilde T ^{-1}(\tilde T(A_k'))\right)=\mu\left(\tilde T^{-1}\left(\bigcup_{k\ge 1}\tilde T(A_k')\right)\right)\\
			&=\nu\left(\bigcup_{k\ge 1}\tilde T(A_k')\right)=\nu\left(\bigcup_{k\ge 1}B_k\right)=\nu(P_2)
		\end{align*} 
		where inequality a) is due to Condition (\ref{eq:extra}). So we have $\mu(P_1)\le \nu(P_2)$, which completes the proof of Condition (\ref{eq:1}).
	\\If by contradiction $y_0\notin P_2$, since $\nu(\{y_0\})>0$ (remember that in our construction $\{y_0\}=B_0$), then 
		\begin{equation}\label{eq:3}
			\nu(P_2)<\nu(B_0\cup P_2)=\mu(P_1),
		\end{equation}
		where in the last equality we have used Condition (\ref{eq:2}). Inequality (\ref{eq:3}) contradicts (\ref{eq:1}) and the \textbf{Claim} is proven. \\
		By the inclusion $y_{0}\in P_{2}$ we now know that $y_{0}$ belongs to $B_{k}$ for some $k\ge 1$. Therefore there exist $x_{k}'\in A'_{k}$ (since $B_k=\tilde T(A_k')$) such that $y_{0}=\tilde T(x_{k}')$.  
		Next we choose $x_k\in A_k$ such that
		\[\max \{c(x_k,\tilde T(x_k')), c(x_k ',T(x_k))\} < \max \{c(x_k,T(x_k)), c(x_k ',\tilde T(x_k ')) \}.\]
		This is possible since $A'_k=\mathcal{R}(A_k)$.
		\\Since by construction $A_k=T^{-1}(B_{k-1})$ and $x_k\in A_k$, we have that $T(x_k)\in B_{k-1}=\tilde T(A_{k-1}')$, so there exists $x_{k-1}'\in A_{k-1}'$ such that $\tilde T(x_{k-1}')\in B_{k-1}$ and  $T(x_k)=\tilde T(x_{k-1}')$.  And so we continue, decreasing with the indices down until having defined $x_1'$ and $x_1$. Since $x_1\in A_1=T^{-1}(B_0)$ we have that $T(x_1)=y_0=\tilde T(x'_k)$. Thus, we have constructed two $k$-uples of points $\left(x_{j}\right)_{j=1}^{k}$ and  $\left(x'_{j}\right)_{j=1}^{k}$ such that for all $j$, $x_j\in A_j$ and $x_j'\in A_j'$.
		\[
		\begin{cases}
			T(x_{j})=\tilde T(x'_{j-1}) &\quad \mbox{for all} \ 2\le j\le k;\\
			T(x_{1})=\tilde T(x'_{k})=y_{0}.
		\end{cases}.
		\]
		and that for all $1\le j\le k$
		\begin{equation}\label{eq:option2}
			\max\{c(x_j',T(x_j)),c(x_j,\tilde T(x_j'))\}<\max\{c(x_j',\tilde T(x_j')),c(x_j,T(x_j))\}.
		\end{equation}
		We observe that by Condition \eqref{eq:option2}, $x_j\not=x_j'$.	
		We now apply to the set $\{x_j\}_{j=1}^k$ the cyclical permutation 
		\[\sigma(j)=\begin{cases}j+1&\text{ if }j\in\{1,\ldots,k-1\}\\
			1&\text{ if }j=k\end{cases}\,.\]
		We have 
		\begin{align}\label{(1)}
			\notag\max_{1\le j\le k} c(x_j,T(x_j))&\stackrel{a)}{\le}\max_{1\le j\le k} c(x_j,T(x_{\sigma(j)}))\stackrel{b)}{=}\max_{1\le j\le k}c(x_j,\tilde T(x_j'))\\
			&\stackrel{c)}{<}\max_{1\le j\le k}\max\{c(x_j,T(x_j)),c(x_j',\tilde T(x_j'))\}\end{align}
		where inequality a) follows from the ICM-condition given by the optimality of $T$. 
		Equality b) follows from the fact that by construction $T(x_{\sigma(j)})=\tilde T(x_j')$ for all $j\in\{1,\ldots, k\}$.
		Estimate c) follows from Condition (\ref{eq:option2}). Concerning the last term in \eqref{(1)}, we have two possibilities: either 
		\[\max_{1\le j\le k}\max\{c(x_j,T(x_j)),c(x_j',\tilde T(x_j'))\}=c(x_m,T(x_m))~~~\text{ for some }m\]
		or 
		\[\max_{1\le j\le k}\max\{c(x_j,T(x_j)),c(x_j',\tilde T(x_j'))\}=\max_{1\le j\le k}c(x_j',\tilde T(x_j'))\,.\]
		The first case leads immediately to a contradiction.  So we are left with the latter case. We apply to the set $\{x_j'\}_{j=1}^k$ the cyclical permutation 
		\[\tau(j)=\begin{cases}j-1&\text{ if }j\in\{2,\ldots, k\}\\
			k&\text{ if }j=1\end{cases}\,.\]
		We get 
		\begin{align}\label{(2)}
			\notag& \max_{1\le j\le k}\max\{c(x_j,T(x_j)),c(x_j',\tilde T(x_j'))\}=
			\max_{1\le j\le k}c(x_j',\tilde T(x_j'))\stackrel{a)}{\le}\max_{1\le j\le k}c(x_j',\tilde T(x_{\tau(j)}'))\\&\stackrel{b) }=\max_{1\le j\le k}c(x_j',T(x_j))
			\stackrel{c)}{<}\max_{1\le j\le k}\max\{c(x_j,T(x_j)),c(x_j',\tilde T(x_j'))\}\,.
		\end{align}
		Above, in a) we have used the ICM property of $\tilde T$, in b) the fact that, by construction,  $T(x_j)=\tilde T(x_{\tau(j)}')$ for all $j\in\{1,\ldots, k\}$, and in c) again Condition (\ref{eq:option2}). Estimate \eqref{(2)} is impossible, so we get a contradiction that concludes the proof.
	\end{proof}
	
	\section{Some examples}
	In this section we give some examples of cost functions that satisfy the assumptions of Theorem \ref{beforemaintheorem}. We also present a counterexample which shows that the assumption of the strict quasiconvexity of $c(\cdot, y)$ is necessary. 
	
	\begin{xpl}\label{jyhlacost}
		Let $c:\R^{d}\times\R^{d}\to\R^{+}$ be a cost function of the type  introduced in \cite{jylha2015}, defined by 
		\begin{equation*}
			c(x,y):=h(y-x),
		\end{equation*}
		where $h:\R^{d}\to\R^{+}$ is a continuous and strictly quasiconvex function. Here, we also assume that there exists a unique $p\in\R^d$ such that $h(p)=0$, that $h$ is differentiable, and that $\nabla h(v)\not=0$ for every $v\not=p$. This cost function satisfies the assumptions of Theorem \ref{beforemaintheorem} from which we infer that the IM transportation plans corresponding to $c$ are of the Monge type. The fact that $c$ fulfils the condition (iv) can be seen as a special case of the proof of the example below. 
	\end{xpl}

	\begin{xpl}
		Let $c:\R^{d}\times\R^{d}\to\R^{+}$ be defined by 
		\begin{equation*}
			c(x,y):=h(G(y)-F(x)),
		\end{equation*}
		where $h:\R^{d}\to\R^{+}$ is defined as in the Example \ref{jyhlacost}, $G:\R^{d}\to \R^{d}$ is an affine transformation of the form \begin{equation*}
			G(y)=Ay+b,
		\end{equation*} where $A\in M^{d\times d}$ is  an invertible matrix and $b\in\R^{d}$ and $F:\R^{d}\to \R^{d}$ has the following properties: \begin{enumerate}[(a)]
			\item  F is differentiable on $\R^{d}$, $F$ is invertible and $\nabla F (x)$ is an invertible matrix for every $x\in \R^d$;
			\item The pre-image $F^{-1}$ of a strictly convex set is a strictly convex set. 
		\end{enumerate}
		Note that affine transformations  $x\mapsto A'x+b'$, where $A'$ is invertible, satisfy Properties (a) and (b), so one class of examples is given by cost functions  of the form $c(x,y)=h(G(y)-F(x))$ where both $F$ and $G$ are affine transformations with invertible matrices.\\
		We observe that the cost function $c$ satisfies the assumptions of Theorem \ref{beforemaintheorem}. By the assumptions on $h$ and the invertibility of $F$ and $G$, Assumption (i) is satisfied. The strict quasiconvexity of $h$, the linearity of $G$ and the Property (b) of $F$ ensure the strict quasiconvexity of $c(\cdot,y)$  for every $y$, and, for the same reason, $c(x, \cdot)$ is strictly quasiconvex for every $x$. 
		Let us prove that the ``twist kind'' condition holds. Assume that there exist $y,\ytilde\in \R^{d}$ and $\lambda>0$ such that $c(x,y)=c(x,\tilde y)=\lambda$ and $n_{c(\cdot,y)}(x)=n_{c(\cdot,\tilde y)}(x)$. We first observe that, by the fact that 
		\[
		\nabla_{x}c(x,y)=-\nabla F(x)^T\cdot\nabla h(G(y)-F(x)),
		\] 
		$\nabla_x c(x,y)\not=0$ for every $x,y$ such that $G(y)-F(x)\not = p$ (that is, for every $x,y$ such that $c(x,y)\not = 0$). Thus, the equality between the two unit normals becomes \[\frac{\nabla_{x}c(x,y)}{|\nabla_{x}c(x,y)|}=\frac{\nabla_{x}c(x,\tilde{y})}{|\nabla_{x}c(x,\tilde{y})|}.\] Since
		\begin{equation*}
			\nabla_{y}c(x,y)=A^T\cdot\nabla h(G(y)-F(x)),
		\end{equation*}
		thanks to the invertibility of $A$ we have that 
		\begin{equation*}
			\nabla_{x}c(x,y)=-\nabla F(x)^T (A^T)^{-1}\nabla_{y}c(x,y).
		\end{equation*}
		The computation above, the assumptions of invertibility on $A$ and $\nabla F(x)$ and the condition on the normalized gradients imply
		\begin{equation*}
			\frac{\nabla_{y}c(x,y)}{|\nabla_{y}c(x,y)|}=\frac{\nabla_{y}c(x,\tilde{y})}{|\nabla_{y}c(x,\tilde{y})|}.
		\end{equation*}
		We observe that both $y$ and $\ytilde$ belong to the boundary of the sublevel set with respect to $x$, $C^{x}_{\lambda}:=\{z\in\R^{d}: c(x,z)\le \lambda\}$. Then $y$ and $\ytilde$ must coincide, since they are two points on the smooth boundary of a strictly convex set $C_\lambda^x$ with the same normal.
	\end{xpl}
	
	The following is an example of an ICM plan that is not given by a map for a cost function that doesn't satisfy the assumptions of Theorem \ref{beforemaintheorem}. 
	\begin{xpl}
		We consider the space $X=\R^2$ with the cost function given by the $\infty$-distance
		\[d((x_1,y_1),(x_2,y_2))=\max\{|x_{1}-y_{1}|,|x_{2}-y_{2}|\}\,.\]
		We denote by $Q:=[0,1]\times[0,1]$, the unit square centered at $\left(\frac12,\frac12\right)$. We take $\mu=\mathcal{L}^2|_{Q}$, the $2$-dimensional Lebesgue measure restricted to $Q$, and $\nu=\mathcal{L}^2|_{Q+(10,0)}$, the $2$-dimensional Lebesgue measure restricted to the translation of $Q$ by the vector $(10,0)$. 
		Let $\gamma\in \Pi(\mu,\nu)$ be defined by 
		\[\gamma=\mu\otimes\mathcal{H}^1|_{\{x_1+10\}\times [0,1]},  \] that is, the transport plan which distributes evenly every point $(x_1,x_2)\in \supp\mu$ to the vertical line segment $\{(x_1+10,y_2)~|~y_2\in \R\}\cap (Q+(10,0))\}$.  
		This plan is clearly not given by a map, but it is ICM.
		
		To prove that $\gamma$ is ICM we fix points 
		$\{((x^i_{1},x^i_2),(x^i_{1}+10,y^i_2))\}_{i=1}^{N}$ in the support of $\gamma$ and a permutation $\sigma$ of the set $\{1,\ldots, N\}$.
		Now 
		\[\max_{1\le i\le N}\{d((x^i_{1},x^i_2),((x^i_{1}+10,y^i_2)) \}=\max_{1\le i\le N}\left\{\max\{10,|x^i_2-y^i_2| \} \right\}=10\]
		where in the last equality we have used the fact that 
		\[|x_2^i-y_2^i|\le 1<10\,.\]
		Let us prove that 
		\[\max_{1\le i\le N}\{d((x^i_{1},x^{i}_2),(x^{\sigma(i)}_{1}+10,y^{\sigma(i)}_2)) \}\geq10\,.\]
		It is enough to show that there exists an index $k\in\{1,\ldots,N\}$ such that 
		\[|x_1^k-(x_1^{\sigma(k)}+10)|\geq 10.\]
		We may assume that $\sigma$ is not the identity: for the identity permutation the ICM condition holds as an equality. 
		Let us denote $I=\{1,\ldots,N\}$ and let $k_0$ be such that 
		\[x_1^{k_0}=\min\{x_1^k~|~k\in I\}.\]
		The point $x_1^{k_0}$ satisfies $x_{1}^{k_0}\le x_1^{\sigma(k_0)}$. Therefore, 
		\[|x_1^{k_0}-(x_1^{\sigma(k_0)}+10)|=10+(x_1^{\sigma(k_0)}-x_1^{k_0})\geq 10\]
		and we are done. 
	\end{xpl}

\section*{Acknowledgements}
The research of the first two authors is partially financed  by the {\it ``Fondi di ricerca di ateneo, ex 60 $\%$''}  of the  University of Firenze and is part of the project  {\it "Alcuni problemi di trasporto ottimo ed applicazioni"}  of the  GNAMPA-INDAM.

The research of the third author is part of the project \emph{Contemporary topics on multi-marginal optimal mass transportation}, funded by the Finnish Postdoctoral Pool (Suomen Kulttuuris\"a\"ati\"o), and of the project \emph{Giovani ricercatori protagonisti 2020}, funded by Fondazione Cassa di Risparmio di Firenze. 

The authors wish to thank the anonymous referee for the help in finding early flows and weakness in the paper and allowing for corrections. 

\bibliographystyle{plain}

\end{document}